\documentclass[12pt]{article}
\usepackage{graphicx}
\usepackage{pgf,tikz,pgfplots}
\pgfplotsset{compat=1.15}
\usepackage{mathrsfs}
\usetikzlibrary{arrows}
\usepackage{amsmath,amssymb,amsthm}
\usepackage [english]{babel}
\theoremstyle{plain}
\newtheorem{thm}{Theorem}[section]
\newtheorem{prop}{Proposition}[thm]
\newtheorem{lem}{Lemma}[thm]
\newtheorem{cor}{Corollary}[thm]
\newtheorem{conjecture}{Conjecture}[thm]

\newcommand{\rth}{\rho_{_{\widehat{T}}}}
\newcommand{\Th}{\widehat{\mathbb{T}}}

\newcommand{\T}{\mathbb{T}}
\newcommand{\rt}{\rho_{_{_T}}}

\newcommand{\R}{\mathbb{R}}
\newcommand{\oph}{\: \widehat{\oplus}\: }
\newcommand{\odh}{\: \widehat{\odot}\: }
\newcommand{\od}{\odot}
\newcommand{\wt}{\widehat{\Theta}}

\begin{document}

\title{Dot Product Representations of Graphs Using Tropical Arithmetic  } \author{Sean Bailey\footnote{Corresponding Author} \footnote{Email: sbailey@tamut.edu} \footnote{Postal Address: 7101 University Ave, Texarkana, TX 75503}\\Department of Mathematics\\Texas A\& M University Texarkana\\ Texarkana, TX 75503 \\
David Brown\footnote{Email: david.e.brown@usu.edu}\\Department of Mathematics and Statistics\\Utah State University\\Logan, Utah 84322-3900
\\Michael Snyder\\Mathematics Department\\ Snow College\\ Ephraim, Utah 84627
\\Nicole Turner\footnote{Email: nturner@lcps.org}\\ Rock Ridge High School\\ Ashburn, Virginia 820148, USA  }

\maketitle

\begin{abstract} 
A dot-product representation of a graph is a mapping of its vertices to  vectors of length $k$ so that vertices are adjacent if and only if the inner product (a.k.a. dot product) of their corresponding vertices exceeds some threshold. Minimizing dimension of the vector space into which the vectors must be mapped is a typical focus. 
We investigate this and structural characterizations of graphs whose dot product representations are mappings into the tropical semi-rings of min-plus and max-plus.
We also observe that the minimum dimension required to represent a graph using a \emph{tropical representation} is equal to the better-known threshold dimension of the graph; that is, the minimum number of subgraphs that are threshold graphs whose union is the graph being represented.
\end{abstract}

\section{Introduction}
The dot product representation (DPR) of a graph was introduced in papers by Reiterman, et al \cite{Reit89} and Schienerman, et al \cite{Fid98} independently. 
Among many others, one motivation for studying DPRs lies in determining the minimum amount of information needed to completely represent a graph. 
In \cite{Fid98} and \cite{Reit89} representing graphs using real-number vector spaces was studied. 
Here we study DPRs in another algebra, the so called tropical semi-ring, and study the minimum dimension 
required for a representation and give structural results and characterizations of graphs that can be represented using  tropical vectors of length $k$, for various values of $k$.

\section{Preliminaries}
\subsection{Basic Definitions}
A \emph{graph} $G$ is an ordered pair $(V,E)$ where $V = V(G)$ is a set of elements called \emph{vertices}, and $E = E(G)$ is a set of 2-element subsets of $V$ called \emph{edges}. Each edge can be seen as a representation of some symmetric relation between it's two constituent vertices, these vertices are referred to as the edge's endpoints. Two vertices that constitute an edge are said to be \emph{adjacent} and are \emph{neighbors} of each other. If two vertices are adjacent
this may be denoted with the juxtaposition of their names. We say a vertex is \emph{incident} to an edge if and only it is one of the endpoints of that edge.

A \emph{path} is a sequence of distinct vertices (except for possibly the first and last) with each two consecutive vertices in the sequence being adjacent. The \emph{length} of a path is the number of edges used in the path. The \emph{distance} between two vertices is the length of the shortest path between them. We will use the notation $P_n$ for a graph that is a path on $n$ vertices. A graph is \emph{connected} if there is a path between every pair of vertices in the graph. 

If the first and last vertices of a path are the same the path is called a \emph{cycle}. A $n-$cycle is a cycle on $n$ vertices, denoted as $C_n$. $C_3$ and $C_4$ are called a \emph{triangle} and a \emph{square} respectively. A \emph{chord} is an edge that connects two nonadjacent vertices of a path or cycle. A graph is called \emph{chordal} if it has no chordless cycles.

Given two graphs, $G = (V,E)$ and $H = (W, F)$, we say $H$ is a \emph{subgraph} of G if $W \subseteq V$ and $F \subseteq E$. A graph generated from another graph by deleting a set of vertices and the edges incident to those vertices is called an \emph{induced subgraph}.

When every pair of vertices in a graph are adjacent we say that the graph is \emph{complete}. A complete graph on $n$ vertices is denoted $K_n$. Every graph is a subgraph of the complete graph on the same number of vertices, but not every graph is an induced subgraph of a complete graph. When a complete graph is a subgraph of a larger graph those vertices form a \emph{clique}, a set of pairwise adjacent vertices. The opposite of a \emph{clique} is an independent set, vertices which are pairwise nonadjacent. The size of the largest independent set in a graph $G$ is denoted as $\alpha(G)$.

The \emph{complement} of a graph $G = (V,E)$ is the graph $\bar{G} = (V, F)$ such that for any $v_i, v_j \in V$ we have $v_iv_j\in F$ if and only if $v_iv_j \notin E$. The \emph{union} of graphs $G_1 = (V_1,E_1)$ and $G_2 = (V_2,E_2)$ is the graph $G_1\cup G_2 = (V_1\cup V_2,E_1\cup E_2)$. The \emph{intersection} of $G_1$ and $G_2$ is the graph $G_1\cap G_2 = (V_1\cap V_2,E_1 \cap E_2)$.

\subsection{Graph Classes}
By analyzing structures in the graph we can classify graphs and introduce parameters that can be compared across different classes. Many classes of graphs deal with the connectivity of the graph, others with the presence or absence of certain structures. We shall review some of these classes and parameters that are relevant to this paper.

A\emph{k-partite graph} is one in which the vertices can be partitioned into $k$ independent (partite) sets such that all edges in $G$ join vertices from different partite sets. A frequently studied class is bipartite graphs, i.e. 2-partite graphs. Because partite graphs require independent sets there can be no complete $k$-partite graphs, $k\ge 2$, using the classical definition of complete; therefore we adopt a modified concept of completeness when discussing them. A $k$-partite graph is considered complete when every pair of vertices not in the same partite set are adjacent. Complete k-partite graphs are denoted similarly to complete graphs, $K_{n_1,n_2,\cdots,n_k}$ , where $n_1, n_2,\cdots,n_k$ are the sizes of the partite sets.  A $K_{1,m}$, with $m \ge 1$, is called a \emph{star}. If $v$ is the vertex in the partite set of size $1$ the star is centered at $v$.

A \emph{tree} is a connected graph with no cycles.A \emph{caterpillar} is a tree with a single path containing at least one endpoint of every edge. A vertex not on this path is called a \emph{leaf}. A star is a tree consisting of a central vertex and $n-1$ leaves adjacent to it. All trees are also bipartite graphs.  Since trees have no cycles they are chordal by default.

Threshold graphs are another class of graphs that are pertinent to the subject of this paper. While there are many characterizations and definitions of threshold graphs
we include only three. For a graph $G = (V,E)$ the following are equivalent:
\begin{enumerate}
\item $G$ is a threshold graph;
\item There are real weights $w(i)$, $i\in V$ , and a threshold value $t$ such that $w(i)+w(j)\ge t$ if and only if $ij \in E$ (We may refer to the function $w : V(G)\to \R$ as the \emph{threshold realization});
\item $G$ can be built sequentially from the empty graph by adding vertices one at a time where each new vertex, when added, is either isolated or adjacent to every existing
vertex;
\item There is no induced subgraph C4, P4, or 2K2 (shown in Figure \ref{fig:threshforbid}).
\begin{figure}[h]
\centering
\includegraphics[scale=1]{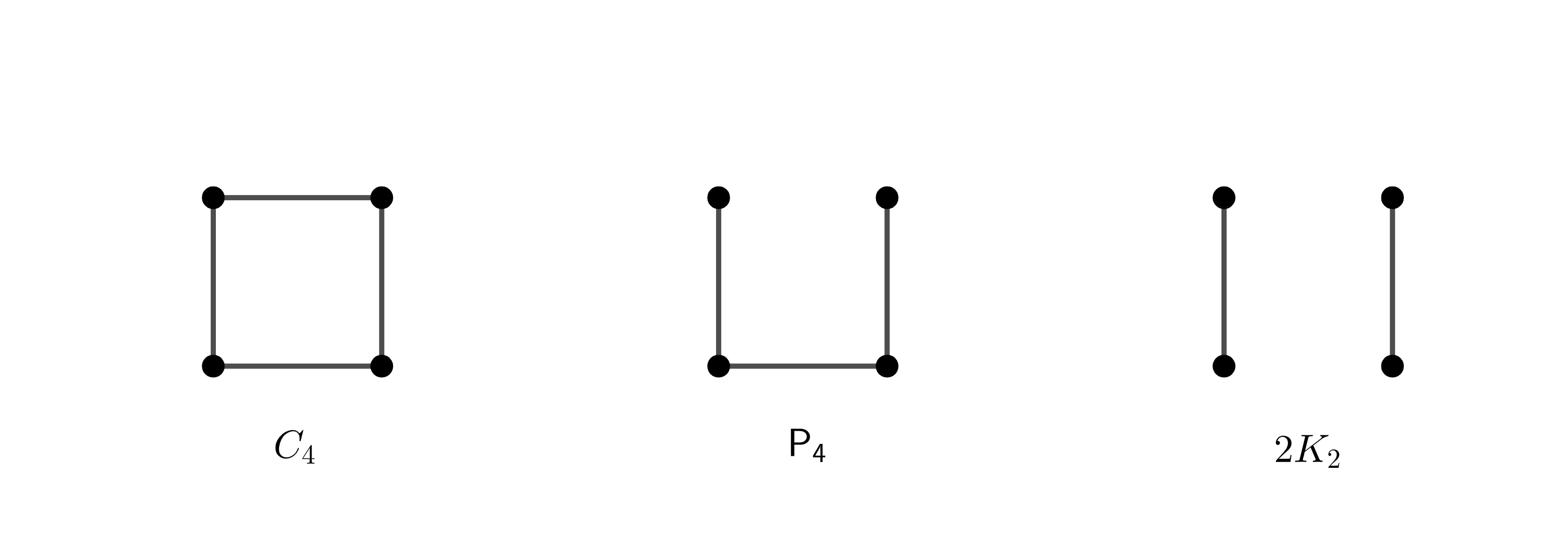}
\caption{Forbidden Induced Subgraphs of Threshold Graphs}
\label{fig:threshforbid}
\end{figure}
\end{enumerate}

In this work we will be primarily concerned with the definition 2 given above. However, definition 4 gives us the result that the complement of a threshold graph is a
threshold graph, since the set $\{C_4, P_4, 2K_2\}$ is closed under complementation.\cite{Reit89}

The \emph{threshold dimension} of $G$ is the minimum number $k$ of threshold subgraphs $T_1,\cdots, T_k$ of $G$ such that the union of these graphs yields G and is denoted $\Theta(G)$. The threshold dimension is well defined, since a single edge along with isolated vertices is a threshold subgraph of $G$, and is bounded above by $|E|$. The problem of finding the threshold dimension of a graph has been shown to be NP-hard and has been of interest for decades, see \cite{Chvatal}. Though generally finding the threshold dimension of a graph is hard it is known for many graphs and classes of graphs. When the threshold dimension
is not known a bound may be established that is based on the structure of the graph, or a well-known parameter such as $\alpha(G)$. Here is one such from \cite{Chvatal} whose proof is given as an example of the arguments made for such results.
\begin{thm}
(\cite{Chvatal}) For every graph $G$ on $n$ vertices we have $\Theta(G)\le n-\alpha(G)$. Furthermore, if $G$ is triangle-free, then $\Theta(G) = n - \alpha(G)$.
\end{thm}

\section{The Dot-Product Representation of a Graph}
Dot product graphs were independently developed by Reiterman et al \cite{Reit89} and Schienerman et al \cite{Fid98}. A graph $G=(V,E)$ is a \emph{$k$-dot product graph} if there exists a function $f:V\to\mathbb{R}^k$ with a real number $t>0$ such that for any $x,y\in V$ $xy\in E$ if and only if $f(x)\cdot f(y)\ge t$. The function $f$ is called a \emph{$k$ dot product representation} of $G$. We call $k$ the \emph{dimension} of the DPR. 

\begin{prop} (\cite{Fid98})
For every graph $G$ there is an integer $k$ so that $G$ is a $k$-dot product graph.
\end{prop} 
We give a more parsimonious representation in Figure \ref{fig:DPGexample}. The vectors can be seen to give a DPR of dimension 2 of the graph.

\begin{figure}[h]
\centering
\includegraphics[scale=1]{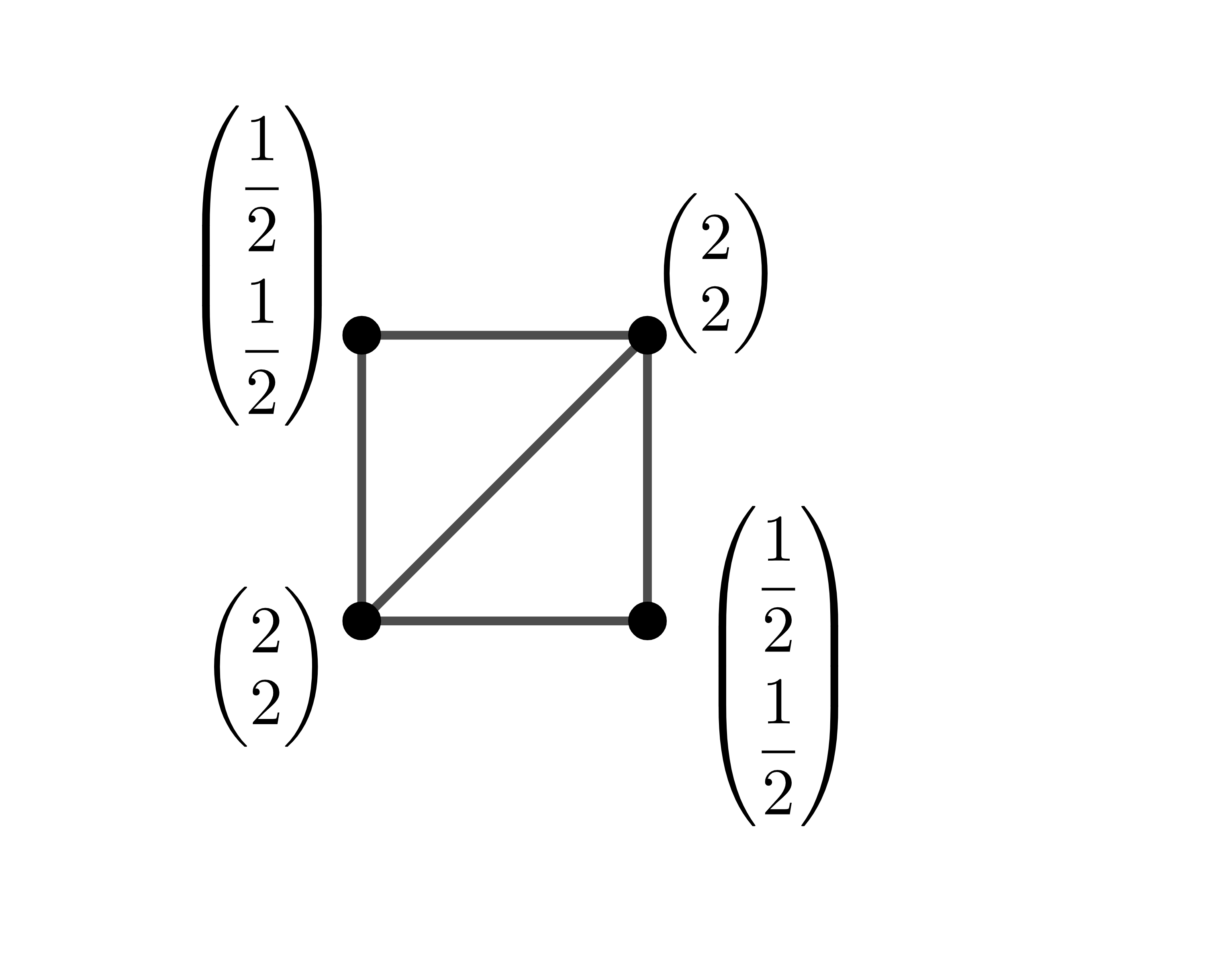}
\caption{A DPR of a Graph}
\label{fig:DPGexample}
\end{figure}

The minimum $k$ such that $G$ is a $k$-dot product graph is called the \emph{dot product dimension} of $G$, $\rho(G)$.

Computing the dot product dimension of a graph has been shown to be $NP$-hard \cite{Ross}, but the dot product dimension for specific classes of structured graphs has been determined. 
The following are some known results about dot product dimension:
\begin{itemize}
\item $\rho(G)\le 1$ if and only if $G$ has at most two non-trivial components that are each threshold graphs. \cite{Fid98}
\item $\rho(C_n)=2$ if $n\ge 4$. \cite{Fid98,Reit89}
\item $\rho(P_n)=2$ if $n\ge4$.  \cite{Fid98,Reit89}
\item If $G$ is an interval graph, then $\rho(G)\le 2$. \cite{Fid98,Reit89}
\item If $G$ is a tree, then $\rho(G)\le 3$. \cite{Fid98}
\item If $G=K_{n,m}$, then $\rho(G)=\min\{n,m\}$. \cite{Fid98}
\item Let $G=(V,E)$, $A\subset V$, and $K_A$ be the clique on $A$. Then $\rho(G\cup K_A)\le\rho(G)+1$. \cite{Reiterman1989}
\end{itemize}

\section{Tropical Arithmetic}
In this paper we consider min-plus and max-plus algebras, often called the tropical semirings. The \textit{min-plus} tropical semring, denoted $\T$, is define as $$(\R\cup \{\infty\}, \oplus, \otimes)$$ with operations $$ x\oplus y := min\{x,y\} \; \; \text{ and }\; \; x \otimes y := x+y.$$ The \textit{max-plus} tropical semiring, denoted $\Th$, is defined similarly $$(\R\cup \{-\infty\}, \oph, \otimes)$$ with $$ x\oph y := max\{x,y\} \; \; \text{ and }\; \; x \otimes y := x+y.$$ 
By letting the operator $\otimes$ take precedence when $\oplus$,$\oph$ and $\otimes$ occur in the same expression,
it is easy to see that associative, commutative, and distributive laws hold in tropical arithmetics.
Both operations also have identity elements with $$x\oplus \infty =x$$ $$x\oph -\infty =x$$ and $$x\otimes 0 =x.$$ 
Tropical division is defined to be classical subtraction so we make no special notation for this operation. 
Since there is no real number, $x$, such that $7\oplus x =11$ has a solution there is no tropical subtraction. 
Thus $\T$ and $\Th$ meet all of the ring axioms except for the existence of additive inverse, whence they are semirings. 
Perhaps not surprisingly $\T$ and $\Th$ are isomorphic \cite{Litv07}, meaning that there is a bijective homomorphism between them. 
That is, there exists some bijective function $f:\T \rightarrow \Th$ (or the other way around) such that for $a,b\in \T$ we have $f(a+b)=f(a)+f(b)$ and $f(ab)=f(a)f(b)$.
Essentially, this means the semirings are the same, up to labeling.

Other mathematical concepts translate naturally over to tropical arithmetic. 
Exponents can be treated as repeated multiplication (i.e. tropical exponentiation is classical multiplication), so in tropical algebras $\sqrt{-1}=-0.5$ because $(-0.5)^2=-0.5\otimes-0.5=-1$.
The operations of vector addition, scalar multiplication, matrix multiplication, and the dot-product of vectors over the tropical semirings are as they are over any field. 
The dot-product operation is of special importance to this work. 
We define the \textit{tropical dot-product} of two vectors of length $n$ as the tropical sum of the coordinate wise tropical products. 
We will use $\od$ to denote the operation of tropical dot product in $\T$ and $\odh$ in $\Th$.
Then given two vectors, $\vec{u}=[u_1, u_2, \ldots, u_n]$ and $\vec{v}=[v_1, v_2, \ldots, v_n]$, we have \begin{align*}
\vec{u} \odot \vec{v} := min\{u_1+v_1, u_2+v_2,\ldots, u_n+v_n\} & \text{ in }\T \\
\vec{u} \odh \vec{v} := max\{u_1+v_1, u_2+v_2,\ldots, u_n+v_n\} & \text{ in }\Th .
\end{align*} 

\noindent Any operation not specifically designated as tropical ($\oplus, \oph, \otimes, \od, \odh$) should be considered as arithmetic over reals. 

\section{Results}

\indent Given the direct translation of the classical dot product to tropical arithmetics the definition of a DPR changes little when done in tropical arithmetics. 
A \textit{$k$-tropical dot product representation of G} is a function $f:V\rightarrow \mathbb{T}^k$ such that, for any $x,y \in V$, $xy \in E$ if and only if $f(x) \odot f(y) \geq t$ or $f(x) \odh f(y) \geq t$. 
We will use the notation $\rho_{T}(G)$ and $\rho_{\widehat{T}}(G)$ to denote the tropical dot product dimension of $G$ using min-plus and max-plus algebras respectively. 
Note that although $\T$ and $\Th$ are isomorphic as rings $\rt$ and $\rth$ differ, but we ultimately show a connection between the two.
  
\begin{thm} Every graph has a min-plus tropical dot product representation.\end{thm}
\begin{proof} Let $G$ be a graph on $n$ vertices and set $t>0$. Assign an ordering to the vertices, $\{v_1, v_2, \ldots v_i, \ldots, v_n\}$. For $f:V\rightarrow \T^n$, build $f(v_i)$ coordinate wise as follows:
\[\text{coordinate } j \text{ of }f(v_i) = \begin{cases}
         \frac t3  & \text{if } j=i\\
         \infty & \text{if }j<i\\ 
         t & \text{if } j > i \text{ and }v_iv_j \in E(G)\\
         \frac t2 & \text{if }j > i \text{ and }v_iv_j\notin E(G) \\
         \end{cases} \] 
Now, for $v_iv_j \in E(G)$ we have $f(v_i) \od f(v_j) = \frac {4t}3 > t$, but if $v_iv_j \notin E(G)$ we have $f(v_i) \od  f(v_j) = \frac{5t}6 <t$ as desired. 
 \end{proof}

\begin{cor} \label{g-v} Let $G$ be a graph and $v\in V(G)$. Then $\rt(G)\le\rt(G-\{v\})+1$.\end{cor}
\begin{proof} Let $k=\rt(G-\{v\})$ and choose $f:V\to \T^k$ be a min-plus $k$-tropical dot product representation of $G-{v}$ with threshold $t>0$. 
Now a min-plus $(k+1)$-tropical dot product representation $\hat{f}$ of $G$ can be formed by adding an extra coordinate to the representation $f$. Let 
 \[
  \hat{f}(u) = \begin{cases}
         \left(\!\!\!\begin{array}{c} f(u)\\ t
         		\end{array} \!\!\!\right)  & \text{when }u \neq v\text{ and }uv \in E(G)\\[.5cm] 
         \left(\!\!\! \begin{array}{c} f(u)\\ \frac{t}{2}
                 \end{array} \!\!\! \right)  & \text{when }u \neq v\text{ and }uv \notin E(G)\\[.5cm] 
         \left(\!\!\! \begin{array}{c} \infty_k \\ \frac{t}{3}
                 \end{array} \!\!\!\right)  & \text{when }u = v \\
         \end{cases}
 \] where $\infty_k$ is a vector of length $k$ with all coordinates equal to $\infty$.  \end{proof}
 
\begin{cor} $\rt(G) \leq n-1$\end{cor}  
\begin{proof} Let $u,v\in V(G)$ such that $uv\in E(G)$.
 Choose $t\geq 0$ and let $f:\{u,v\}\rightarrow \R$ be the mapping $f(u)=t$ and $f(v)=t$.
  Recursively apply Corollary \ref{g-v} to build a $(n-1)$-dot product representation of $G$ in $\T$.\end{proof}
 
\begin{thm} Every graph has a max-plus tropical dot product representation. \end{thm}
\begin{proof}  Let $G$ be a graph on $n$ vertices and let the threshold be $t\ge 0$. Suppose that the A $n$-tropical dot product representation may be created for an ordering to the vertices, $\{v_1, v_2, \ldots v_i, \ldots, v_n\}$. For $f:V\rightarrow Th^n$, build $f(v_i)$ coordinate wise as follows:
\[\text{coordinate } j \text{ of }f(v_i) = \begin{cases}
         t  & \text{if } j=i\\ 
         \frac t3 & \text{if } v_iv_j\in E(G) \\
         0 & otherwise
         \end{cases} \] 
Now, for $v_iv_j \in E(G)$ we have $f(v_i) \od f(v_j)  =\frac {4t}3 >t $, but if $v_iv_j \notin E(G)$ we have $f(v_i) \od  f(v_j) \le \frac {2t}3 <t$ as desired.         
 \end{proof}

\indent We have proven the existence of these tropical dot product representations using an arbitrary, $t$. When given the graph and are left to determine the vectors we find that the threshold does not influence the value of the tropical threshold dimension.
\begin{thm} The threshold, $t$, of any tropical dot product representation of a graph $G$ can be any arbitrary $t>0$.\end{thm}
\begin{proof} Let $f:V\to \R^k$ be a tropical $k$-dot product representation of $G$ with threshold $t^*$. Then $\hat{f}:V\to \R^k$, $\hat{f}(u)=\frac{t}{t^*}f(u)$, is a tropical $k$-dot product representation of $G$ with threshold $t$ \end{proof}

\indent Since our choice of threshold does not affect the tropical dot product dimension of a graph, from here on we will put the threshold $t=1$ unless otherwise noted.

\section{Results with $\Theta (G)$}
\indent Note that for vectors $u, v$ of length one $v \od u = v \odh u$, since the minimum value of a set of size one is the same as the maximum value of that set. Hence if $G$ can be represented in $\T$ we have $\rt(G)=\rth(G)$. Thus the following theorem holds also for $\rth(G)$.
\begin{thm} \label{1} Let $G$ be a graph. Then $\rt(G)=1$ if and only if $G$ is a threshold graph.\end{thm}
\begin{proof} $\rt(G)=1 \Leftrightarrow (f:V\to \R$ such that $f(v)\odot f(u)\ge t \Leftrightarrow vu\in E) \Leftrightarrow (f(v)\od f(u) \geq t \Leftrightarrow f(v)+f(u)\geq t) \Leftrightarrow$ $G$ is a threshold graph.\end{proof}

\indent This result implies that any graph with tropical dot product dimension 1 can have at most one non-trivial component. This is different than the result for $\rho(G)=1$, where $G$ could have 2 components that are non-trivial threshold graphs. Thus a $2K_2$ is an example where $\rho(G)$ differs from  $\rt(G)$ and $\rth(G)$.\\ 
\indent Theorem \ref{1} also led us to consider how our tropical dot product dimensions of a graph $G$, $\rho_{{T}}(G)$ and $\rho_{\widehat{T}}(G)$, correlate with the theshold dimension of $G$, $\Theta(G)$. 

\begin{thm}\label{threshdim} $\rho_{\widehat{T}}(G) = \Theta (G)$ \end{thm}
\begin{proof} $\rho_{\widehat{T}}(G) \leq \Theta(G)$: Let $t=1$ and $G$ be a graph such that $\Theta(G)=m$. Let $\{G_1, G_2, \ldots, G_m\}$ be a set of threshold graphs that realize $\Theta(G)$. Let $g_i:V(G_i) \rightarrow \R$ be the threshold realization of $G_i$. Create $f(v_i)$ as follows: 
\[
 \text{coordinate }j\text{ of }f(v_i) = \begin{cases}
         g_j(v_i) & \text{when } v_i \in V(G_i)\\
        0  & otherwise
        \end{cases}
\] 
Thus we have an max-plus $m$-tropical dot product representation of $G$ and $\rho_{\widehat{T}}(G)\leq \Theta(G)$.\\
\indent $\rho_{\widehat{T}}(G) \geq \Theta(G)$: Let $(G)$ be a graph such that $\rho_{\widehat{T}}(G)=k$. This means there is a max-plus $k$-tropical dot product representation of $G$, call it $f$. By Theorem \ref{1} we see that if we consider the graph created by the dot product of vectors made of only the value of the $j^{th}$ coordinate of $f(v_i)$ for all $v_i\in G$ we have a threshold graph that is a subgraph of $G$. By this consideration we find $k$ threshold graphs such that the union of these graphs gives us $G$. By the definition of threshold dimension we know that $\Theta(G)\leq k$ and hence $\Theta(G) \leq \rho_{\widehat{T}}(G)$\\
Thus $\rho_{\widehat{T}}(G)=\Theta(G)$. \end{proof}

\indent This result gives us a wealth of free information about $\rth(G)$ and also gives us a new way to explore the well researched parameter $\Theta$. Problems and results dealing with $\Theta(G)$ can now be rephrased in terms of $\rth (G)$. For example: \\
\begin{cor} For every graph $G$ on n vertices we have $\rth (G) \leq n- \alpha (G)$. Furthermore, if $G$ is triangle-free, then $\rth (G) = n- \alpha (G)$. \end{cor} 
\indent Define the \textit{threshold graph intersection number} as the minimum value for $k$ such that there exists a set of $k$ threshold graphs, ${G_1, G_2, \ldots, G_k}$, with  $G_i=(V,E(G_i))$ and $G=\cap_{i=1}^{k}G_i$. 
We denote the threshold graph intersection number as $\wt(G)$. The intersection of two graphs cannot have more edges than either of the original graphs. 
This hints at a connection between our new parameter and $\rt(G)$.

\begin{thm} For $G=(V,E)$, $\rt (G)=\wt(G)$. \end{thm}
\begin{proof} \textbf{$\rt(G)\geq \wt(G)$}: Suppose $\rt(G)=k$, and $f:V\rightarrow R^k$ is a $k$-tropical dot product representation of $G$. Let $v_i$ represent the value of the $i^{th}$ coordinate of $f(v)$.
 Create a set of $G_i=(V,E(G_i))$ which are the threshold graphs formed by $f_i(v)=v_i$ for $v\in V$ and $t=1$. If $xy\in E(G)$ then $x \od y \geq 1$ so $x_i+y_i \geq 1$ for all $1 \leq i \leq k$. 
 Hence $xy \in E(G_i)$ for $1 \leq i \leq k$ and $xy\in \cap_{i=1}^k E(G_i)$. If $xy\notin E(G)$ then $x \od y < 1$ and there exists some value of $i$ for which $x_i + y_i < 1$. 
 For this $i$, $xy\notin E(G_i)$ and hence $xy \notin \cap_{i=1}^k G_i$. Thus $G=\cap_{i=1}^k G_i$ and $\rt(G)\geq \wt(G)$.\\
\indent \textbf{$\rt(G) \leq \wt(G)$}: Let $X=\{G_1, G_2, \ldots, G_k\}$ be a set of threshold graphs on $V$ such that $G=\cap_{i=1}^k G_i$. Since each $G_i$ is a threshold graph, let the function $f_i:V\rightarrow R$ be a mapping that realizes $G_i$ as a threshold graph. Create $f:V \rightarrow R^k$ where for $v\in V$  the $i^{th}$ coordinate of $f(v)$ is $f_i(v)$. Since $G=\cap_{i=1}^k G_i$, for $xy\in E(G)$ we know that $xy\in E(G_i)$ for $1 \leq i \leq k$ so $f_i(x)+f_i(y) \geq 1$ and by definition $f(x)\od f(y) \geq 1$ as desired. If $xy\notin E(G)$ there is a $G_i$ with $xy\notin E(G_i)$. Then $f_i(x)+f_i(y) <1$ and $f(x)\od f(y) <1$. Thus $f$ is a tropical $k$-dot product of $G$ and $\rt \leq \wt(G)$.\\
\noindent Therefore $\rt(G)= \wt(G)$.\end{proof}

\indent Utilizing de Morgan's Law from set theory, we are now able to make a connection between $\rt(G)$ and $\rth(G)$ using de Morgan's law and their connections to $\Theta(G)$ and $\wt(G)$.
 
\begin{thm}\label{bigc} $\wt(G)=\rt(G)=\rth(\overline{G})=\Theta(\overline{G})$\end{thm} 
\begin{proof} If $G$ is a threshold graph then $\overline{G}$ is a threshold graph.\cite{Reit89} Thus $\rt(G)=\rth(G)=1$.\\
If $G$ is not a threshold graph then suppose $\rt(G)=k$. Then there exist graphs $G_1, G_2, \ldots, G_k$, $G_i=(V,E(G_i)$, such that $G=\cap_{i=1}^kG_i$. By de Morgan's Law $\overline{G}=\overline{\cap_{i=1}^kG_i}=\cup_{i=1}^k\overline{G_i}$ and so $\rt(G) \geq \rth(\overline{G})$.\\
\indent For $\rt(G)\leq \rth(\overline{G})$ suppose $\rth(\overline{G})=k$ then $\overline{G}=\cup_{i=1}^kG_i$. By de Morgan's Law $G = \overline{\cup_{i=1}^kG_i}= \cap_{i=1}^k(\overline{G})$ so $\rt(G)\leq \rth(\overline{G})$.\end{proof}

\indent The following corollary is a direct result of this theorem.
\begin{cor}\label{comp} If $G$ is self complementary then $\rt(G)=\rth(G)$.\end{cor}

\section{Other Results}
\indent Beginning with the results above we can establish values and bounds on the tropical threshold dimensions for non-threshold graphs. For many of the cases below the bound or value of the max-plus tropical dot product dimension has already been determined by those studying threshold dimension and can be found in resources such as \cite{threshold}. As the purpose of this work is not to reiterate these findings, we predominantly give results about $\rt(G)$.

\begin{cor} $\rt(C_4)=2$\end{cor}
\begin{proof} The complement of a $C_4$ is a $2K_2$. This graph has threshold dimension 2, so $\rth(\overline{C_4})=2$ by Theorem \ref{threshdim}. Then by Theorem \ref{bigc} $\rt(C_4)=2$. \end{proof}

\begin{cor}\label{cor:2k2} $\rt(2K_2)=2$\end{cor}
\begin{proof} The complement of a $2K_2$ is a $C_4$ which has threshold dimension 2. Thus by Theorem \ref{threshdim} $\rth(\overline{2K_2})=2$ so $\rt(2K_2)=2$ by Theorem \ref{bigc}.\end{proof}

\begin{cor} $\rt(P_4)=2$\end{cor}
\begin{proof} A $P_4$ is self complementary and has threshold dimension 2. Thus by Theorem \ref{threshdim} and Corollary \ref{comp} $\rth(P_4)=\rt(P_4)=2$. \end{proof}

Though the threshold dimension of a graph, or its complement, generally may not be known, we can still say some things about the min-plus tropical dot product dimension of it. Many of these results echo the type of results that have been obtained regarding $\rho(G)$, some of which are listed in Section 2.

\begin{cor} For any graph $G$ on $n \geq 3$ vertices, let $H$ be the largest induced subgraph of $G$ that is a threshold graph. If $H$ has $k$ vertices then $\rt(G)\le n-k+1$.\end{cor}
\begin{proof} Since any graph on $n=3$ vertices is a threshold graph, every graph has an induced subgraph that is a threshold graph on at least 3 vertices. Find the largest such graph, call it $H$. By Theorem \ref{1} there exists a function $f$ that is a 1-tropical dot product representation of this subgraph. Now apply Lemma \ref{g-v} $n-(k-1)$ times to create $\hat{f}$ an $n-k+1$ tropical dot product representation of $G$. \end{proof}

A result for $\rho(G)$ is that by adding edges to $G$ to form a clique on some subset of vertices of $G$ the dot product dimension is increased by at most 1. The \textit{join} of two graphs, $G_1=(V_1,E_1)$ and $G_2=(V_2,E_2)$, is the graph $G_1 \vee G_2 = G_1 \cup G_2 \cup K_{V_1, V_2}$, where $K_{V_1, V_2}$ is the bipartite graph with $V_1$ as one partite set and $V_2$ as the other. In other words, $G_1 \vee G_2$ adds an edge from every vertex of $G_1$ to every vertex of $G_2$. When $G_2$ is a complete graph this operation does not increase the tropical dot product dimension of a graph: 

\begin{lem} \label{vee} Let $G$ be a graph and $K_n$ a complete graph on $n$-vertices. Then $\rt(G\vee K_n)=\rt(G)$.\end{lem}
\begin{proof} Let $\rt(G)=k$ and $f:V(G) \to \R^k$ be a tropical $k$-dot product representation of $G$. Let $H = G \vee K_n$ and let $g:V(H) \to \R^k$ be the mapping where \[
 g(v_i) = \begin{cases}
        f(v_i)  & \text{when } v_i \in V(G)\\
        \vec{1}_k   & \text{when }v_i \in K_n \\
        \end{cases}
\] 
where $(1_k)$ is the vector of entries equal to 1 of length $k$. Now the vertices of the $K_n$ are adjacent to all vertices in $H$, but the adjacencies between vertices originally in $G$ remain as they were. \end{proof}

Note that this same proof works to show that $\rth(G\vee K_n)=\rth(G)$.

\begin{thm} Let $G$ be a complete $k$-partite graph. Then $\rt(G)\le k$.\end{thm}
\begin{proof} Let $t=1$. Consider the function $f:V \to \R^k$ that maps vertices in partite set $i$, $1 \leq i \leq k$, to the vector with entry $i$ equal to 0 and all other entries equal to 1. Suppose $v_p$ and $v_q$ are in different partite sets; then $f(v_p) \odot f(v_q)= 1$. Suppose $v_p$ and $v_q$ are in the same partite set; then $f(v_p) \odot f(v_q)= 0$. Thus we have a tropical $k$-dot product representation of a complete $k$-partite graph. \end{proof}

If given only a little more information about the sizes of the partite sets we improve the results.

\begin{cor}\label{parts} Let $G$ be a complete $k$-partite graph with the size of each partite set greater than 1. Then $\rt(G)=k$.\end{cor}
\begin{proof} Consider $\overline{G}$ which consists of $k$ components that are complete graphs. The threshold dimension of a disconnected graph must be at least the number of components, and exactly that number when each component is a threshold graph. Since a complete graph is a threshold graph we have $\Theta(\overline{G})=k$ and hence $\rt(G)=k$ by Theorem \ref{bigc}.\end{proof}

\begin{cor} Let $G$ be a complete $k$-partite graph with $m$ $(m<k)$ of the partite sets of cardinality 1. Then $\rt(G)=k-m$.\end{cor}
\begin{proof} Let $\hat{G}$ be the graph induced on the $k-m$ partite sets of size greater than 1. By Corollary \ref{parts} $\rt (\hat{G}) = k-m$. Now apply Lemma \ref{vee} $m$ times to add $m$ $K_1$s to $\hat{G}$ to have a tropical $(k-m)$-dot product representation of $G$. \end{proof}

From Corollary \ref{cor:2k2} that $\rho(2K_2)< \rt(2K_2)$. Consider now $K_{m,n}$ where $n,m>2$. From Corollary \ref{parts} we know $\rt(K_{m,n})=2$ but $\rho(K_{m,n})=min\{m,n\}>2$. Thus we have examples of $\rho(G)\leq\rt(G)$ and $\rho(G)\geq\rt(G)$ so we cannot use one as a bound of the other. In \cite{Reit89} it was shown that $\rho(G)\leq \Theta(G)$ so $\rho(G)\leq \rth(G)$ for any graph.

It was also shown in \cite{Reit89} that if $T$ is a tree then $\rho(T)\leq 3$. A bound on $\rt$ for trees in general has not been established, but we have been able to establish the same bound for $\rt$ of caterpillars as for $\rho$ in \cite{Reit89}:

\begin{thm}\label{cat} Let $G$ be a caterpillar. Then $\rt(G)\le 2$.\end{thm}
\begin{proof} Proof by construction: First we will construct a min-plus dot product representation of a path of longest length. Find such a path and label the vertices $p_1, p_2, \ldots, p_m$ beginning at one end and proceeding down the path. Let $v_{i_j}$ be the $j^{th}$ leaf attached at vertex $p_i$. A caterpillar labeled in this way is shown in Figure \ref{fig:caterpillar}. Choose $k\geq 2$ and let $d_i= \lfloor \frac{i-1}2\rfloor$ for vertex $p_i$, $1 \leq i \leq m$. Then give each vertex a function value as follows:
\[
 f(p_i) = \begin{cases}
        \left(\!\!\!\begin{array}{c} \frac{1}{k+d_i}\\ \frac{k+d_i}{k+d_i+1}
                		\end{array} \!\!\!\right) & \text{when } i \text{ is odd}\\ \\
       \left(\!\!\!\begin{array}{c}\frac{k+d_i}{k+d_i+1}\\ \frac{1}{k+d_i+1}
               		\end{array} \!\!\!\right) & \text{when } i \text{ is even}\\
        \end{cases}
\] 

Consider vertices $p_j$ and $p_s$, and without loss of generality let $j < s$. If they are adjacent then by construction $f(p_j)\odot f(p_s)\geq 1$. If they are not adjacent and both $j$ and $s$ are odd then the sum of the first coordinates will be less than one, if they're both even the sum of the second coordinates will be less than 1. If they are nonadjacent and $j$ is odd and $s$ is even then $s-j \geq 3$ so $d_j<d_s$ and the sum of the second coordinates of the vectors will give $\frac{k+d_j}{k+d_j+1}+\frac{1}{k+d_s+1} < 1$ since $\frac{1}{k+d_s+1}<\frac{1}{k+d_j+1}$. Similarly if $j$ is even and $s$ is odd then the sum of the first coordinates is less than 1. For all these cases $f(p_j) \odot f(p_s) <1$ as desired.
\begin{figure}[h]
\centering
\includegraphics[scale=0.6]{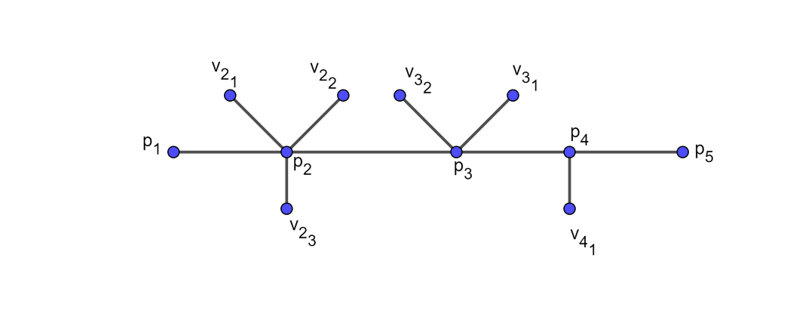}
\caption{A caterpillar with its vertices labeled}
\label{fig:caterpillar}
\end{figure}

Now for $v_{i_j}$, 
$$f(v_{i_j})= \left(\!\!\!\begin{array}{c} 1\\ 1
                		\end{array} \!\!\!\right)-f(p_i)$$
Figure \ref{fig:caterpillarrep} shows a $2$-dot product representation of $B$ using this construction with $k=2$.
\begin{figure}[h]
\centering
\includegraphics[scale=0.5]{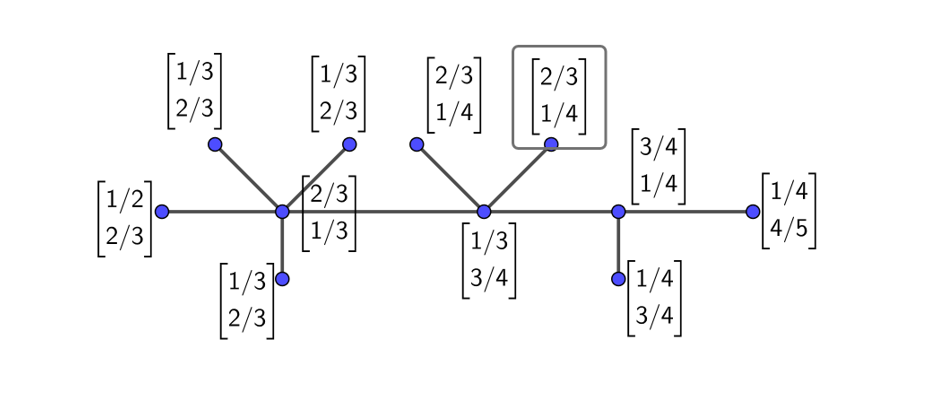}
\caption{A min-plus $2$-dot product representation of $B$}
\label{fig:caterpillarrep}
\end{figure}

By construction these leaves are adjacent to $p_i$. There are 2 other types of vertices it should not be adjacent to. In each case the progressively increasing denominator and the alternating nature of the vector coordinates cause the tropical dot product to fail to reach 1. In all these cases either the sum of one set of coordinates is of two values less than $\frac 12$ or of a value of more than $\frac 12$ and a value smaller than needed to add to 1. The following tables detail each of these cases and tells which of the sets of coordinates fails to add to 1. First consider $f(v_{i_j}) \odot f(v_{s_s})$
\[ \begin{array}{c|c|c}
i & s& \text{coordinates whose sum is } <1\\
even & even & 1^{st}\\
odd & even & 1^{st}\\
odd & odd & 2^{nd}\\
even & odd & 2^{nd}\\
\end{array}
\]
Now consider $f(v_{i_j}) \odot f(p_{s})$
\[ \begin{array}{c|c|c|c}
i & s& \text{index constraints} & \text{coordinates whose sum is } <1\\
even & odd & i\neq k & 1^{st}\\
odd & even & i\neq k & 2^{nd}\\
even & even & i < k & 2^{nd}\\
 & & i > k & 1^{st}\\
odd & odd & i < k & 1^{st}\\
 & & i>k & 2^{nd}\\
\end{array}
\]
Thus $f:V\rightarrow \R^2$ as described above produces a tropical dot product representation of dimension 2, and $\rt(G) \leq 2$ for any caterpillar. $\rt (G)=2$ when $G$ is a caterpillar that is not a threshold graph. \end{proof}

This also gives a value for non threshold subgraphs of caterpillars, i.e. paths.

\begin{cor}\label{path} $\rt(P_n)=2$ for $n>3$.\end{cor}
\begin{proof} Since $P_n$ is a caterpillar, by Theorem \ref{cat} and Theorem \ref{1} $\rt (P_n)=2$ when $n>3$. \end{proof}

Theorem \ref{cat} also reduces the min-plus dot product dimension upper bound for any graph that has an induced subgraph
 that is a caterpillar which is at least 2 vertices larger than it's largest induced subgraph that is a threshold graph.\\

\begin{cor} \label{cycle} $\rt(C_n)\leq 3$. \end{cor}
\begin{proof} If $n\leq 4$ the tropical dot product dimension has been shown. Let $n\geq 5$. For $v\in V$, $C_n-v$ is a path, which has tropical dot product dimension 2. Then by Lemma \ref{g-v} $\rt(C_n)\leq \rt(C_n-v)+1$ hence $\rt(C_n)\leq 3$.\end{proof}

Because of the linear nature of caterpillars we are able to use the labeling system in Theorem \ref{cat} on disconnected graphs where each component is a caterpillar.

\begin{cor} Let $G$ be a graph with $k>1$ components that are caterpillars. Then $\rt(G)=2$. \end{cor}
\begin{proof} Find a path of longest length for each caterpillar. Beginning with one caterpillar, label the vertices following the pattern from Theorem \ref{cat}. When a caterpillar is completely labeled continue the labeling on another caterpillar beginning the path labels with 2 more than the highest path label already used. 
By design the construction in Theorem \ref{cat} connects path vertex $p_i$ only to $p_{i-1}$ and $p_{i+1}$. Thus the last vertex of one caterpillar and the first of another will not be adjacent, and the rest of the labeling provides only the desired adjacencies as shown in the proof of the theorem above. \end{proof}

We have already seen that $\rho$ does not necessarily relate to $\rt$ and is a lower bound for $\rth$. It is yet to be determined if there is a bounding relationship between the min-plus and max-plus dot product dimensions of a graph.

\begin{thm} There exist graphs such that $\rt(G) \neq \rho_{\widehat{T}}(G)$ \end{thm} 
\begin{proof} $\Theta(P_6)=3 \Rightarrow \rho_{\widehat{T}}(P_6)=3$ but $\rt (P_6)=2$ by Corollary \ref{path}. \end{proof}

Consequently, if there is to be a bounding of one by the other, we have only one option.
\begin{conjecture} $\rt(G) \leq \rth(G)$ for all $G$. \end{conjecture}

\section{Applications}
Applications of tropical mathematics have become more prevalent in past couple of decades. We will introduce two applications of tropical dot product graphs.

\subsection{A Min-Plus Tropical Dot Product Graph Application}
Our first application is an education application for teachers pairing students. Suppose we have a class of $n$ students and we want to pair/group the students such that they have a minimal combined competency in $k$ key skill areas. For each student, the instructor can rate each student in the key skill areas on a scale of $0-3$. Each student is rated for each skill area as follows: 0 - the student does no skills in this area, 1 - the student has some skills in this area, 2 - the student has most of the skills in this area, and 3 - the student has all of the skills in this area. Using this ratings, each student can be assign a vector with $k$ coordinates, where each coordinate corresponds to a specific key skill area. We can now choose a threshold $t$, which is the desired combined competency of students. A min-plus tropical dot product graph can be created using the student vectors and given threshold. To further demonstrate this application, consider the following example.

Consider the math class with 6 students whom we will call $A$, $B$, $C$, $D$, $E$, and $F$. The students are rated in the areas of 1) ability to solve polynomials, 2) ability to simplify algebraic expressions, 3) ability to take derivatives of functions, and 4) ability to take integrals of functions. Suppose that based on our assessments, our students are assigned the following vectors:
\begin{center}
$A=\begin{bmatrix}
2\\2\\1\\0
\end{bmatrix}$\hspace{1em}$B=\begin{bmatrix}
1\\2\\2\\1
\end{bmatrix}$\hspace{1em}$C=\begin{bmatrix}
1\\1\\3\\3
\end{bmatrix}$\hspace{1em}$D=\begin{bmatrix}
3\\3\\0\\0
\end{bmatrix}$\hspace{1em}$E=\begin{bmatrix}
2\\1\\3\\2
\end{bmatrix}$\hspace{1em}$F=\begin{bmatrix}
1\\2\\2\\3
\end{bmatrix}$
\end{center}
For our threshold, we will assign $t=3$. This specific threshold (combined competency) of $3$ will require at least one of each pair to have a score of at least $2$. In other words, at least one student will have most of the skills in each of the areas. 

The resulting graph of this representation would be Figure \ref{fig:tropapp1}.

\begin{figure}[ht]
\centering
\includegraphics[scale=.75]{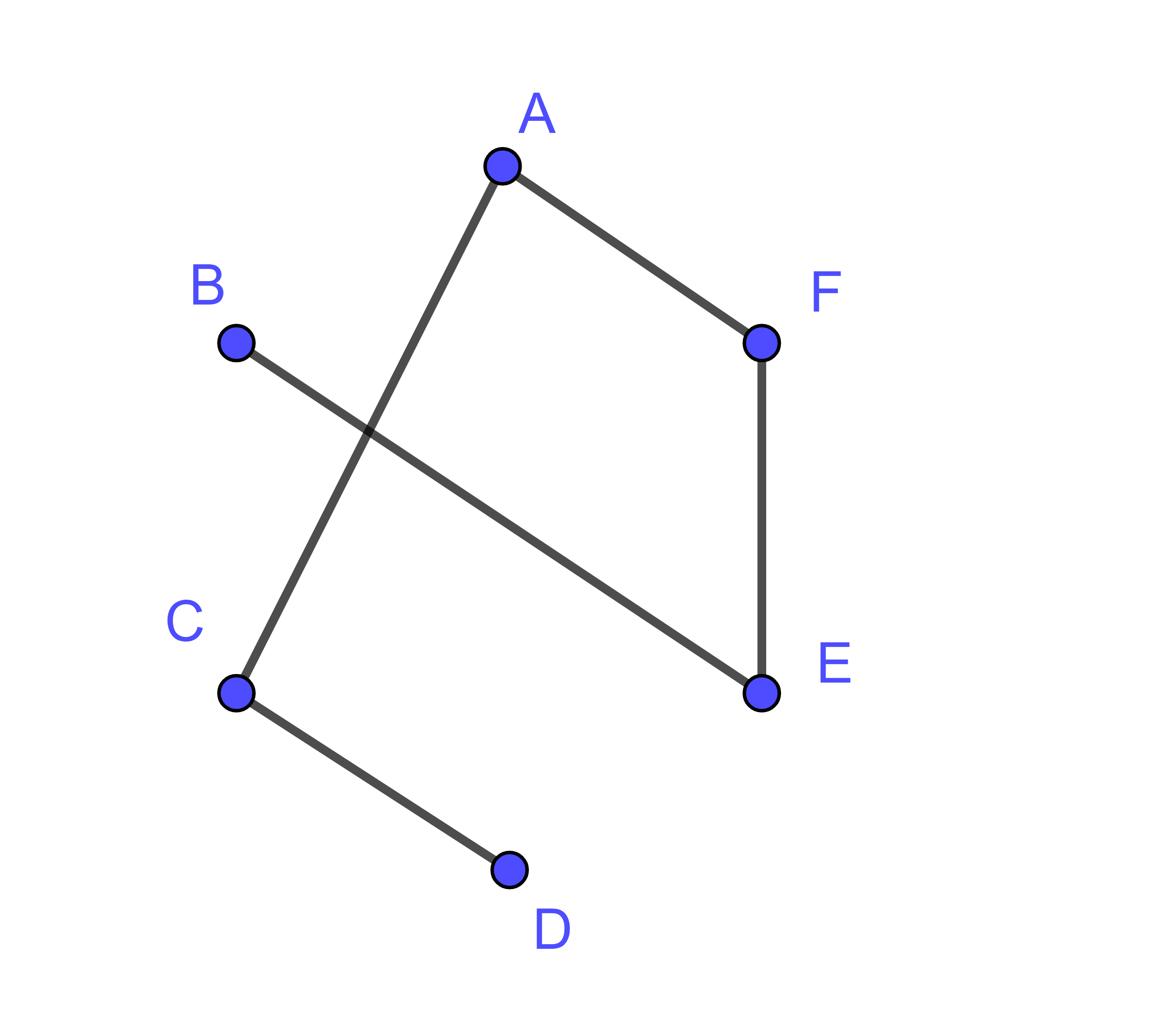}
\caption{Tropical Dot Product Graph Generated For Students}
\label{fig:tropapp1}
\end{figure}

Using Figure \ref{fig:tropapp1}, we can assign pairs by starting with starting with a vertex of minimum degree. $B$ is one such vertex in our graph. $B$ can only be paired with $E$, so our first pair is $B$ and $E$. Since $E$ has a pair, $F$ must be paired with a different student. $A$ can be paired with $F$.  Now that we have the pairs $\{B, E\}$ and $\{A,F\}$, we can examine the last two students, namely $C$ and $D$. Since the vertices are adjacent, we can pair these students. Thus our class can have the pairs $\{B, E\}$, $\{A,F\}$, and $\{C,D\}$.

This application was designated for education, but it could also be used in business to pair employees on projects.

\subsection{A Max-Plus Tropical Dot Product Graph Application}
Our second application is for investing in mutual funds. Suppose we have $n$ mutual funds we are interested in investing it. A mutual fund is a investment portfolio of different assets such as stocks, bonds, and other securities. For our specific mutual funds, we can identify $k$ key stocks, bonds, or securities of interest. For each mutual fund, we can assign a vector with $k$ coordinates, where coordinate $i$ corresponds to investment $i$ and is given the value of $1$ is the mutual fund includes that investment and $0$ otherwise. The max-plus dot product of two vectors will be $2$ if the two respective mutual funds share at least one investment. Thus the max-plus tropical dot product graph generated by these given vectors and a threshold of $2$ will help identify mutual funds both with and without common investments. To show this application, consider the following example.

We will consider 7 mutual funds $\{A, B, C, D, E, F, H\}$. The securities being considered are 1) Stock $\alpha$, 2) Stock $\beta$, 3) Bond $\delta$, 4) Bond $\gamma$, and 5) Precious Metal $\psi$. The assigned vectors would be 
\begin{center}
$A=\begin{bmatrix}
1\\0\\0\\1\\0
\end{bmatrix}$\hspace{0.75em}$B=\begin{bmatrix}
1\\1\\0\\0\\0
\end{bmatrix}$\hspace{0.75em}$C=\begin{bmatrix}
0\\0\\1\\1\\0
\end{bmatrix}$\hspace{0.75em}$D=\begin{bmatrix}
0\\0\\1\\1\\1
\end{bmatrix}$\hspace{0.75em}$E=\begin{bmatrix}
0\\1\\1\\0\\1
\end{bmatrix}$\hspace{0.75em}$F=\begin{bmatrix}
1\\0\\0\\0\\1
\end{bmatrix}$\hspace{0.75em}$H=\begin{bmatrix}
0\\1\\0\\1\\0
\end{bmatrix}$
\end{center}
The resulting graph of this representation would be Figure \ref{fig:tropapp2}.

\begin{figure}[ht]
\centering
\includegraphics[scale=.75]{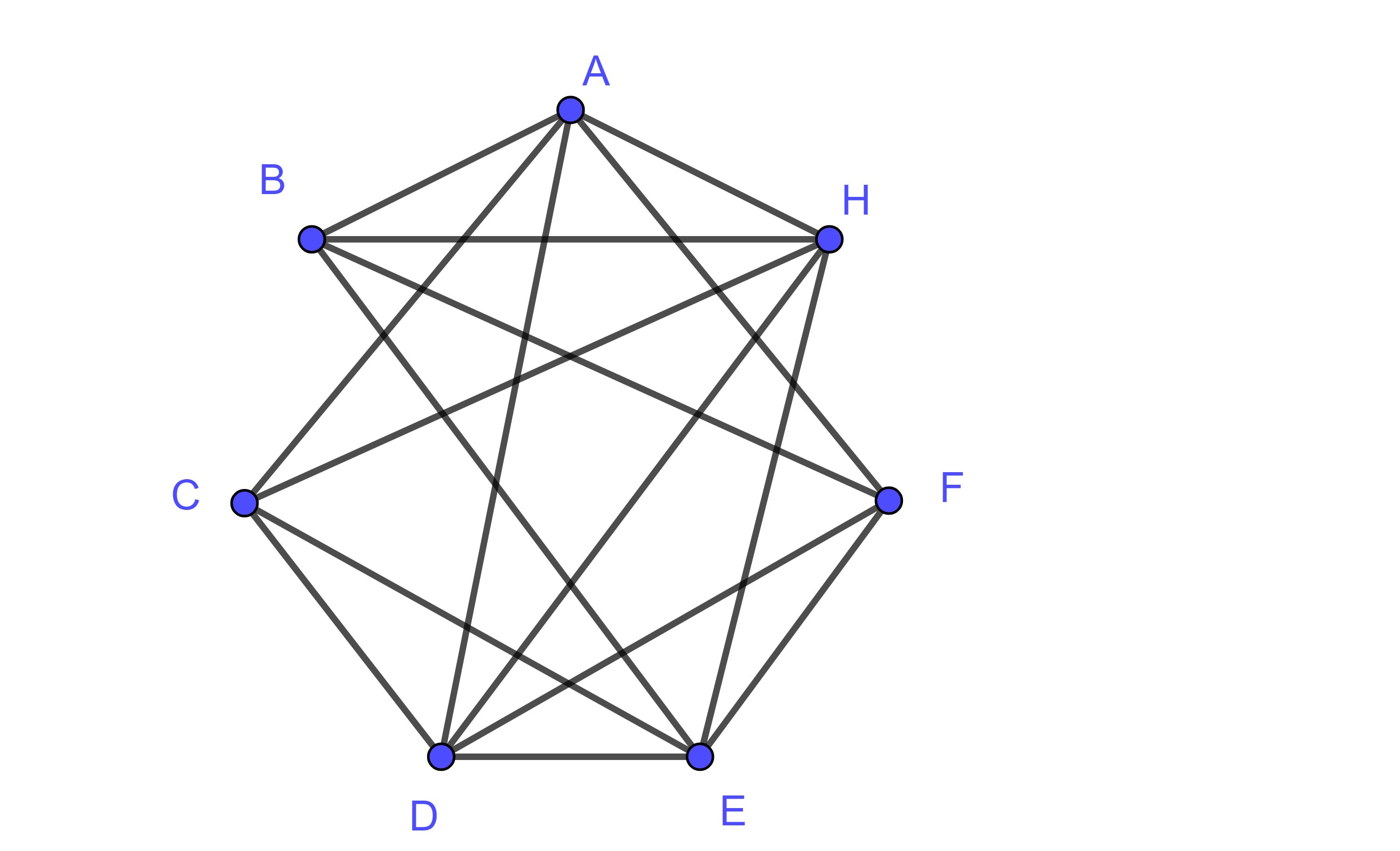}
\caption{Tropical Dot Product Graph Generated For Mutual Funds}
\label{fig:tropapp2}
\end{figure}

Examining Figure \ref{fig:tropapp1}, we can see that one of the maximum clique is $\{A, C, D, H\}$. This set of mutual funds could be invested in if you want to similar profiles. Thus these mutual funds would show some similar growths. Conversely, $\{A,E\}$ is a maximum independent set. The lack of adjacency implies that this set of mutual funds would have no similar investments for a more diverse profile.

\section{Conclusion}

\indent Much remains to be determined about tropical dot product graphs and their dimensions. We conclude this work with a list of open questions.
\begin{enumerate}
\item Is $\rt(G)\leq \rth(G)$ for all graphs?
\item Can we determine $\rth(G)$ of graphs for which $\Theta(G)$ is unknown? %\\
\item Characterize $k$-dot product graphs under $\T$. %\\
\item Does $\rt(C_n)=3$ for all $n\geq 5$? %\\
\item Does $\rt(T)$, where $T$ is any tree, have an upper bound?
\item We know threshold graphs are a subclass of interval graphs. Is there an upperbound for $\rt(G)$ where $G$ is an interval graph?
\item Is there a connection between other graph properties and $\rt$? If so, what is that connection?
\end{enumerate}

\bibliographystyle{abbrv}

\end{document}